\documentclass[12pt]{article}

\usepackage{latexsym,amsfonts,amsmath,epsfig,tabularx,amsthm,dsfont}

\theoremstyle{definition}

\newtheorem{thm}{Theorem} 

\newtheorem{lem}[thm]{Lemma}
\newtheorem{prop}[thm]{Proposition}

\newcommand{\N}{\mathbb{N}}
\newcommand{\nat}{\mathbb{N}} 
\newcommand{\R}{\mathbb{R}}

\newcommand{\abs}[1]{\left\vert #1 \right\vert}	
\newcommand{\norm}[1]{\left\Vert #1 \right\Vert}	

\renewcommand{\a}{\alpha}

\newcommand{\eps}{\varepsilon}

\newcommand\dint{{\rm d}}

\title{On Weak Tractability of the Clenshaw--Curtis Smolyak Algorithm}

\author{Aicke Hinrichs\\
Institut f\"ur Mathematik, Universit\"at Rostock\\
Ulmenstra\ss e 69, Haus 3, 18051 Rostock, Germany\\
email: aicke.hinrichs@uni-rostock.de\\
\qquad
\\
Erich Novak\footnote{This
author was partially supported by the DFG-Priority Program 1324.}, 
Mario Ullrich\footnote{This author was partially supported by 
DFG GRK 1523.}\\
Mathematisches Institut, Universit\"at Jena\\
Ernst-Abbe-Platz 2, 07743 Jena, Germany\\
email: erich.novak@uni-jena.de, 
ullrich.mario@gmail.com}

\begin{document}

\maketitle

\begin{abstract}
We consider the problem of integration of $d$-variate analytic functions 
defined on the unit cube 
with directional derivatives of all
orders bounded by 1.
We prove that the Clenshaw Curtis Smolyak algorithm  
leads to weak tractability of the problem.
This seems to be the first positive tractability result for the Smolyak algorithm 
for a normalized and unweighted problem. 
The space of integrands is not a tensor product space 
and therefore we have to develop a different proof technique. 
We use the polynomial exactness of the algorithm as well as an 
explicit bound on the 
operator norm of the algorithm.
\end{abstract}

\section{Introduction and Result}   \label{sec1}  

We prove that the Clenshaw Curtis Smolyak algorithm is weakly 
tractable for a class of analytic functions. 
Weak tractability of the integration problem for this class
was recently shown in \cite{HNUW13}. In [12]  high
derivatives are approximated by finite differences.
This approximation is very unstable, it does not give
a practical algorithm. In this paper we show with different
proof techniques that the Smolyak algorithm can
be used with essentially the same error bounds. Therefore 
we are now able to give a constructive algorithm, while the
result from \cite{HNUW13} was only a complexity
result.

To explain our result in detail, we say a few words about this algorithm, 
about recent tractability results, and about our proof technique. 

\subsection{The CCS algorithm}

We want to compute
$$
S_d(f) = \int_{[0,1]^d} f(x) \, {\rm d} x
$$
and use the Smolyak algorithm~\cite{Sm63}
in combination with the Clenshaw Curtis algorithm 
as in Novak, Ritter~\cite{NR96,NR97,NR99}, 
see also
Gerstner, Griebel~\cite{GG98} 
and 
Petras~\cite{Pe00}.
We describe the resulting CCS algorithm. 

For $f\colon[0,1]\to\R$ define the sequence of (one-dimensional) 
quadrature rules
\begin{equation} \label{eq:U}
U^\ell(f) \;=\; \sum_{j=1}^{m_\ell} a_j^\ell\, f(x_j^\ell), \quad \ell\in\N, 
\end{equation}
with
\[
m_\ell \;=\; \begin{cases}
1, & \ell=1 \\
2^{\ell-1}+1, & \ell\ge2 . 
\end{cases}
\]
For $\ell=1$ there is only one node $x_1^1=1/2$ with weight $a_1^1=1$.
For $\ell > 1$ we have the 
nodes
\[
x_j^\ell \;=\; \frac12 \left(1-\cos\Bigl(\pi\frac{j-1}{m_\ell-1}\Bigr)\right), 
\qquad j=1,\dots,m_\ell
\]
and weights
\begin{equation} \label{eq:a}
a_j^\ell \;=\; \begin{cases} 
\frac{1}{2m_\ell(m_\ell-2)}, & j=1,m_\ell\\
\frac{1}{m_\ell-1} \left( 1 - \frac{\cos (j-1)\pi}{(m_\ell-1)^2-1}  
				- 2\sum\limits_{k=1}^{\frac{m_\ell-3}{2}} 
\frac{\cos \frac{2k(j-1)\pi}{m_\ell-1}}{4 k^2 -1}\right), & j=2,\dots m_\ell-1.
\end{cases}
\end{equation}
These rules are called Clenshaw-Curtis (CC) quadrature rules. 
It is well known that the CC-rules are positive rules, 
that is $a_j^\ell>0$ for all $j$ and $\ell$, 
see~\cite{BP11}. 

Observe that the nodes of the  $U^\ell$ are nested, since
\[
x_{2j-1}^{\ell+1} = x_j^{\ell}  \qquad \mbox{for }\  j=1,2,\dots m_\ell.
\]

Additionally, we define
\begin{equation} \label{eq:Delta}
\Delta^1 \,=\, U^1 \quad \text{ and } \quad \Delta^\ell \,=\, U^\ell-U^{\ell-1}, 
\quad \ell\ge2.
\end{equation}
Note that, for $f\colon[0,1]\to\R$,  
\[
\Delta^\ell(f) \;=\; \sum_{j=1}^{m_\ell} b_j^\ell\, f(x_j^\ell), \quad \ell\in\N, 
\]
with
\[
b_{j}^\ell \;=\; \begin{cases}
a_j^\ell & \text{ for even } j \\ 
a_j^\ell - a_{\frac{j+1}{2}}^{\ell-1}  & \text{ for odd } j. 
\end{cases}
\]
These weights, except for $\ell=1$, sum up to zero.

Then the Smolyak algorithm (based on the CC rule $U^\ell$) is defined by 
\[
A(q,d) \;=\; \sum_{i\in\N_0^d\colon |i|\le q-d} 
	\Delta^{i_1+1} \otimes\dots\otimes \Delta^{i_d+1}.
\]
Here, the $d$-fold tensor product of the functionals $\Delta^\ell$ 
is given by
\[
     (\Delta^{\ell_1} \otimes\ldots\otimes \Delta^{\ell_d})(f) 
\;=\; \sum_{j_1=1}^{m_{\ell_1}}\ldots \sum_{j_d=1}^{m_{\ell_d}}\,
	b_{j_1}^{\ell_1}\cdot\ldots\cdot b_{j_d}^{\ell_d}\; 
	f(x_{j_1}^{\ell_1},\ldots,x_{j_d}^{\ell_d})
\]
for $f\colon[0,1]^d\to\R$.
Note that we can write this 
CCS algorithm $A(q,d)$, for $q=d+k$, as
\[
A(d+k,d) \;=\; \sum_{\ell=1}^{k+1} A(d+k-\ell,d-1) \otimes \Delta^\ell.
\]
The Smolyak algorithm can also be written as 
$$
A(q,d) = 
\sum_{q-d+1 \le |i| \le q} (-1)^{q-|i|} \cdot 
\binom{d-1}{q-|i| } \cdot
(U^{i_1} \otimes \dots \otimes U^{i_d}) ,
$$
see Wasilkowski and Wo\'zniakowski~\cite[Lemma 1]{WW95}.
Clearly $A(q,d)$ is a linear functional, and $A(q,d)(f)$
depends on $f$ only through function values at a finite
number of points. To describe these points let
$$
X^i = \{ x_1^i, \dots, x^i_{m_i} \} \subset [0,1]
$$
denote the set of nodes of $U^i$. The tensor product
algorithm $U^{i_1} \otimes \dots \otimes U^{i_d}$ is based on
the grid $X^{i_1} \times \dots \times X^{i_d}$, and therefore
$A(q,d)(f)$ depends (at most) on the values of $f$ at points in the union
$$
H(q,d) = \bigcup_{q-d+1 \leq |i| \leq q} 
(X^{i_1} \times \dots \times X^{i_d}) \subset [0,1]^d
$$
of grids. Nested sets $X^i \subset X^{i+1}$
yield $H(q,d) \subset H(q+1,d)$ and
$$ 
H(q,d) = \bigcup_{|i| = q} 
(X^{i_1} \times \dots \times X^{i_d}).
$$ 
Therefore nested sets seem to be the most economical choice. 
The points $x \in H(q,d)$ are called hyperbolic cross points and
$H(q,d)$ is also called a sparse grid.  

In what follows we will bound the number of function values that 
are sufficient for the CCS algorithm to achieve a certain error.
For this we define
\[
N_d(k) \;:=\; \bigl|H(d+k,d)\bigr|
\]
as the number of points used by $A(d+k,d)$.

\subsection{Some known properties of the CCS algorithm} 

Error bounds for the Smolyak algorithm were proved by
Smolyak~\cite{Sm63}, 
Wasilkowski and Wo\'zniakowski~\cite{WW95}
and many others, see also \cite{BG04,Gr06,SU06,SU11}. 
In this paper we always consider the worst case 
error with respect to the unit ball for some norm,
and therefore properties of various norms 
are very relevant.

Most of the known error bounds are for tensor product spaces, i.e., 
one takes norms with 
\begin{equation}    \label{prop-tens} 
\Vert f_1 \otimes f_2 \otimes \dots \otimes f_d \Vert = 
\prod_{i=1}^d \Vert f_i \Vert, 
\end{equation} 
where
$$ 
(f_1 \otimes f_2 \otimes \dots \otimes f_d)(x) = 
\prod_{i=1}^d  f_i(x_i) .
$$ 
We stress that in this paper we do \emph{not} use tensor product norms 
since we use the norm
\begin{equation}   \label{norm} 
\sup_{k \in \N_0}  \sup_{\theta\in S^{d-1}}\, 
\|D^k_\theta f\|_\infty , 
\end{equation} 
where $D_\theta f$ denotes the directional derivative of $f$ in direction 
$\theta$.
Therefore we cannot use the property \eqref{prop-tens} 
and error bounds based on it. 
We want to illustrate this a bit further. 
For
$f_1(x) = x$ on $[0,1]$
clearly all partial derivatives of $f_1 \otimes f_1 \otimes \dots \otimes 
f_1$ are bounded by 1.
Some directional derivatives are larger than 1 and hence, 
for the norm \eqref{norm},  
$$
\Vert f_1 \otimes f_1 \Vert > \Vert f_1 \Vert \cdot \Vert f_1 \Vert. 
$$
This property makes the unit ball with respect to 
the norm \eqref{norm} \emph{smaller} than the unit 
ball of a tensor product space. 

Consider now $C^r([0,1]^d)$ with the standard norm 
\begin{equation}    \label{standard} 
\Vert f \Vert = \max_{|\alpha| \le r} \Vert D^\alpha f \Vert_\infty, 
\end{equation} 
where $\alpha = (\a_1, \dots , \a_d) \in \N_0^d$ is used to denote a 
partial derivative of order $|\alpha | = \a_1 + \dots + \a_d$. 
Then, for $r=1$ and
$f_2(x)=x^2$ on $[0,1]$,
$$
2 = \Vert f_2 \otimes f_2 \Vert < \Vert f_2 \Vert \cdot \Vert f_2 \Vert = 4. 
$$
This property makes the unit balls of $C^k([0,1]^d)$ \emph{larger} than 
the unit balls of tensor product spaces. 
We present  
a result from \cite{NR97} for the CCS algorithm for the order of convergence
and the standard norm \eqref{standard}, hence  the unit ball is 
$$ 
C^r_d = \{ f: [0,1]^d \to \R \mid 
\max_ {|\alpha| \leq r} \Vert D^\alpha  f \Vert_\infty \le 1 \}.
$$

\begin{prop}  \label{prop1} 
For $d, r \in \N$ there exists $c_{r,d}>0$ such that
\begin{equation} 
e(A(q,d),C^r_d) = \sup_{f \in C^r_d} | A(q,d)(f) - S_d(f) | 
\leq  c_{r,d} \cdot N^{-r/d} \cdot(\log N)^{(d-1) \cdot (r/d +1)},
\end{equation} 
where $N=N_d(q-d)$ is the number of function values used by the CCS algorithm. 
\end{prop} 

We describe the proof since in this paper we use a similar technique,
but with a different emphasis. 
Observe that Proposition~\ref{prop1} 
contains unknown constants $c_{r,d}$ and hence the error 
bound makes sense only for given $r$ and $d$ and very large $N$ 
or small error $\eps$. For the proof we need three facts. 

First we need 
an estimate of the number $N_d(k)$ of knots that are used by
$A(d+k,d)$. We use $\approx$ to denote the strong equivalence of
sequences, i.e., $v_n \approx w_n$ iff $\lim_{n\to \infty }v_n/w_n=1$.
Then, 
for $k \to \infty$ and fixed $d$,  
\begin{equation}   \label{TMG} 
N_d(k) \approx \frac{1}{(d-1)! \cdot 2^{d-1}} 
\cdot 2^k \, k^{d-1}, 
\end{equation} 
see M\"uller-Gronbach~\cite[Lemma 1]{MG}.

Smolyak's construction leads to cubature formulas with negative
weights, even if positive weights are used in the univariate case.
However, the weights are relatively small in absolute value. 
Hence, secondly, we need a bound for $\norm{A(q,d)}$, 
where $\norm{A(q,d)}$ denotes the sum of the 
absolute values of the weights of $A(q,d)$. 
We use $c_d$ and $c_{r,d}$ to denote different positive constants
depending on $d$ or on $r$ and $d$, respectively.
There exists a constant $c_d>0$ such that
\begin{equation}    \label{norm von CCS} 
\norm{A(d+k,d)} \leq c_d \cdot  \Big(\log \big( N_d(k)\big)\Big)^{d-1} .
\end{equation} 

For a proof observe that
$$
\# \big\{ i \in \N^d \,\big|\, |i| = \ell \big\} = 
\binom{\ell - 1}{d-1}.
$$
Since the Clenshaw-Curtis formulas have positive weights, we
conclude
\begin{equation}  \label{norm2} 
\norm{A(d+k,d)}
\leq c_d \cdot \! \!  \sum_{\ell=k+1}^{d+k} \binom{\ell - 1}{d-1} \leq
c_d \cdot \binom{d+k - 1}{d-1} \leq c_d \cdot (d+k)^{d-1}, 
\end{equation} 
where we can take $c_d = 2^d$. 
Then \eqref{norm von CCS} follows from 
$$
\log N_d(k) \geq c_d \cdot (d+k). 
$$

The third fact that we need is that 
$A(d+k,d)$ is exact for all polynomials of total 
degree at most $2k+1$, see~\cite{NR96,NR99}. \qed 

\medskip 

We add in passing that the estimates 
\eqref{TMG}, \eqref{norm von CCS} and \eqref{norm2} are not 
suitable for tractability studies. 
In particular we cannot use estimates that contain unknown or 
exponentially large constants $c_d$.  

\subsection{The curse of dimensionality} 

\newcommand{\points}{\mathcal{P}} 

We study multivariate integration 
for different classes $F_d$ of smooth functions  
$f \colon [0,1]^d \to \R$.  
Our emphasis is on large values of $d\in\N$. 
We want to approximate 
\begin{equation}   \label{int} 
S_d(f) = \int_{D_d} f(x) \, \dint x \quad \mbox{for} \quad  
f\in F_d 
\end{equation} 
up to some error $\eps >0$.
 
We consider (deterministic)  
algorithms that use only function values, and   
classes $F_d$ of functions bounded in  
absolute value by 1 and containing all constant functions 
$f(x)\equiv c$ with $|c|\le 1$.  
An algorithm that uses no function value at all must be a constant,  
$A_0(f)\equiv b$, and its error is at least 
$$ 
\max_{f\in F_d}|S_d(f)|=1. 
$$ 
We call this the initial error of the problem, it does not depend on $d$.  
Hence multivariate integration is well scaled and that is why 
we consider $\eps<1$.   
 
Let  
$n(\eps,F_d)$  
denote the minimal number of  
function values needed for this task in the worst case  
setting.
By the \emph{curse of dimensionality} we mean that  
$n(\eps,F_d)$ is exponentially large in $d$.  
That is, there are positive numbers $c$, $\eps_0$ and $\gamma$ such that 
\begin{equation}                    \label{curse} 
n(\eps,F_d) \ge c \, (1+\gamma)^d \quad  
\mbox{for all} \quad \eps \le \eps_0  \quad  
\mbox{and infinitely many} \quad d\in \nat.  
\end{equation} 
For many natural classes  $F_d$ the bound in~\eqref{curse} 
will hold for all  
$d\in\nat$. 
There are many classes $F_d$ for which the curse of dimensionality has  
been proved, 
see~\cite{NW08,NW10} for such examples. 

The classes $C^r_d$  
were already studied in 1959 by  
Bakhvalov~\cite{Ba59}, see also~\cite{No88}. 
He proved that there are two positive  
numbers $c_{r,d}$ and $\tilde c_{r,d}$ such that 
\begin{equation}\label{bak1959} 
c_{r,d}\,\eps^{-d/r}\le n(\eps, C^r_d )\le 
\tilde c_{r,d} \,\eps^{-d/r} \quad 
\text{for all}  \ d\in\nat 
\ \text{and} \ \eps\in(0,1). 
\end{equation} 
This means that for a fixed $d$ and for $\eps$ tending to zero,  
we know that $n(\eps, C^r_d )$ is of order $\eps^{-d/r}$  
and the exponent of $\eps^{-1}$  
grows linearly in $d$.  
If we compare this with Proposition~\ref{prop1}  
we may say 
that the CCS algorithm is ``almost optimal'' 
for each class $C^r_d$. In this sense the algorithm is 
``universal''.

Bakhvalov's result does not allow us to 
conclude whether the curse of dimensionality holds for the classes $C^r_d$. 
In fact, if we reverse the roles of $d$ and $\eps$, and consider a fixed $\eps$ 
and~$d$ tending to infinity, the bound~\eqref{bak1959} 
on $n(\eps, C^r_d )$ is useless. 
The curse of dimensionality for the classes $C^r_d$ 
was only recently proved 
in Hinrichs, Novak, Ullrich and Wo\'zniakowski~\cite{HNUW12}. 
 
\begin{prop}                              \label{cor:Lip} 
The curse of dimensionality holds for the classes  
$C^r_d$ with the  
\emph{super-exponen\-tial} lower bound 
\[ 
n(\eps,C^r_d ) \ge c_r\,(1-\eps) \, d^{\,d /(2r+3)} 
\]  
for all $d\in\nat$ and $\eps\in(0,1)$, 
where $c_r\in(0,1]$  
depends only on $r$. 
\end{prop} 
 
One may say that the classes $C^r_d$ are too large and therefore 
we obtain the curse of dimensionality. 
Therefore it is natural to study smaller classes such as the 
unit balls $F_d$ with respect to the norm \eqref{norm}. 
In 
Hinrichs, Novak, Ullrich and Wo\'zniakowski~\cite{HNUW13}
we prove 

\begin{prop}                              \label{weak}
The curse of dimensionality does not hold for the classes  $F_d$
since the problem is weakly tractable, i.e., 
$$
\lim_{d+\eps^{-1} \to \infty} 
\frac{\log (n(\eps, F_d ))}{d+\eps^{-1}} = 0.
$$
\end{prop}

This means that, for a fixed $\eps$, the complexity 
of integration is sub-exponential in the dimension. 
Unfortunately, the proof of Proposition~\ref{weak} 
in \cite{HNUW13} 
is rather theoretical, we use a very unstable 
algorithm 
which is based on the approximation of high derivatives by
function values via finite differences, 
see also Vyb\'iral~\cite{Vyb13}. 
This algorithm 
could not be implemented
because of huge rounding errors. 
The aim of this paper is to give a much more constructive 
proof of Proposition~\ref{weak} by means of the CCS algorithm, 
see Theorem~\ref{thm:main}.  

\subsection{Main result} 

Let
\begin{equation} \label{eq:class}
F_d = \Bigl\{ f\in C^\infty([0,1]^d) \mid 
\sup_{k \in \nat_0}  \sup_{\theta\in S^{d-1}}\, 
	\|D^k_\theta f\|_\infty\le 1 \Bigr\}, 
\end{equation}
where $D_\theta f$ denotes the directional derivative of $f$ in direction 
$\theta$.

\begin{thm} \label{thm:main}
For each $d\in\N$ and $\eps\in(0,1]$ define 
\[
k_{\eps,d} \;:=\; 
\biggl\lceil\max\Bigl\{3\,d^{2/3},\; \ln(1/\eps)\Bigr\} \biggr\rceil .
\]
Then,
\[
e(k_{\eps,d}, d) \;:=\; 
\sup_{f\in F_d} \abs{A\bigl(d+k_{\eps,d},d\bigr)(f) \,-\,  S_d(f)} \;\le\; \eps
\]
and the number of function values $N_d(k_{\eps,d})$ used by 
the CCS algorithm 
$A(d+k_{\eps,d},d)$ satisfies
\[
N_d(k_{\eps,d}) \;\le\; 2\,\exp\Bigl\{k_{\eps,d}\, 
	\Bigl( 2 + \ln\Bigl( 1- d/\ln(\eps) \Bigr) \Bigr)\Bigr\}.
\]
\end{thm}

This shows, in particular,  that the 
problem of integration for $F_d$ is weakly tractable 
and that the CCS algorithm is weakly tractable 
for these classes. 

One may argue that also the CCS algorithm is ``mildly unstable'' 
and one would prefer an algorithm with small  
operator norm, such as 
a              cubature formula with positive weights that add up to 1.
Indeed, we prove an
analogue of Theorem \ref{thm:main} with a better 
dependence of the number of nodes.

\begin{thm} \label{thm:qmc}
For each $d\in\N$ and $\eps\in(0,1]$ define 
\[
k^*_{\eps,d} \;:=\; 
\biggl\lceil\max\Bigl\{4\,d^{1/2},\; \ln(1/\eps)\Bigr\} \biggr\rceil .
\]
Then there exists a           cubature rule $Q\bigl(k^*_{\eps,d},d\bigr)$ 
with positive weights that add up to 1 with error 
\[
e(k^*_{\eps,d}, d) \;:=\; 
\sup_{f\in F_d} \abs{Q\bigl(k^*_{\eps,d},d\bigr)(f) \,-\,  S_d(f)} \;\le\; \eps
\]
and the number of function values $N^*_d(k^*_{\eps,d})$ 
used by $Q\bigl(k^*_{\eps,d},d\bigr)$ satisfies
\[
N^*_d(k^*_{\eps,d}) \;\le\; \exp\Bigl\{k^*_{\eps,d}\, 
	\Bigl( 1 + \ln\Bigl( 1- d/\ln(\eps) \Bigr) \Bigr)\Bigr\}.
\]
\end{thm}

The proof is based on a constructive version of Tchakalov's Theorem 
due to Davis, see \cite{Dav67}. However, to construct these cubature formulas, one
has to solve exponentially (in $d$) many linear systems of equations, 
each having exponentially many unknowns.
So these methods can be applied only for small $d$.
In contrast, the CCS Smolyak algorithm can be easily implemented.

\subsection{Related results and open problems} 

\begin{itemize} 

\item 
Considering the above remarks about the relation of Theorems
\ref{thm:main} and \ref{thm:qmc}, a natural question is
whether the weak tractability 
of integration for $F_d$ can be proved with a 
positive cubature formula 
which can be efficiently constructed. 
Additionally, we pose the same question for QMC algorithms, 
i.e.~positive cubature formulas with equal weights. 
For recent surveys on QMC algorithms see \cite{DKS,DP}. 

\item
The classes 
$$
F^d = 
\{ f: [0,1]^d \to \R \mid 
\Vert D^\alpha f \Vert_\infty \le 1 \ \hbox{for all} \ 
\alpha \in \nat_0^d   \} 
$$
were studied several times in the literature,
also for the $L_p$ approximation problem, see 
\cite{HNUW12,HNUW13,HZ07,NW08,NW09b,NW10,Wei,Woj03}. 
Here we only mention that $F_d$ from this paper is smaller 
than $F^d$ and it is still \emph{not} known whether integration 
is weakly tractable for the classes $F^d$. 

\item
We do not know whether integration for the classes $F_d$ 
from \eqref{eq:class}
is uniformly weakly tractable. See Siedlecki~\cite{Sied12} 
for this stronger notion of tractability.  

\item
There is an algorithm for the approximation problem
that uses the same sparse grid $H(q,d)$ as well as interpolation by 
polynomials, see \cite{BNR98}. 
This algorithm is often applied, see, e.g., \cite{BNT} 
and one may ask about tractability properties of this algorithm. 
We do not know whether the $L_p$ approximation problem for the classes $F_d$ 
from \eqref{eq:class}
is weakly tractable and, in 
particular, whether the weak tractability follows from properties 
of the Smolyak algorithm. 

\item
The Smolyak algorithm was generalized to the weighted tensor 
product algorithm by Wasilkowski and Wo\'zniakowski~\cite{WWW1,WWW2} 
and these authors also proved tractability results for weighted 
tensor product problems; see also \cite{NW10}. 

\end{itemize} 

\section{The proof}

We start with computing the norms of $\Delta^\ell$ 
and note that a similar (slightly weaker) result 
was already proved by Petras~\cite{Pe00}.  

\begin{lem}         \label{lemma:Delta}
For $\Delta^\ell$ from \eqref{eq:Delta} we have 
\[
\|\Delta^1\|=1, \quad \|\Delta^2\|=\frac23\quad \text{ and }\quad 
\|\Delta^\ell\|=1+\frac1{4^{\ell-1}+1}
\]
for $\ell\ge3$.
\end{lem}

\begin{proof}
Recall that the norm of a quadrature rule is given by the 
sum of the absolute values of the used weights. 
Obviously, $\Delta^1$ has only one weight equal to 1, so $\| \Delta^1 \|=1$.
For $\Delta^2$ it is an easy computation to check that there are 
three nodes with weights 
$b_1^2 = \frac16, b_2^2 = -\frac13, b_3^2=\frac16$, which
gives 
\[
\| \Delta^2 \| \;=\; \sum_{j=1}^3 |b_j^1| \;=\; \frac23.
\]
We now treat the case $\ell>2$. Since we want to sum up the absolute values of 
the weights $b_j^\ell$ we first consider their signs. 
Clearly, $b_{j}^\ell=a_j^\ell>0$ for even $j$. 

For odd $j$ we now show that $b_{j}^\ell<0$ for all $\ell\ge3$ and $j\le m_\ell$.
That is, we show that 
\[
a_{2j-1}^{\ell+1} < a_j^{\ell}
\]
for $\ell\ge2$ and (not necessarily odd) $j=1,\dots,m_{\ell}$. 
For $j=1,m_{\ell}$,
noting that $m_{\ell+1}=2 m_{\ell}-1$,
this follows immediately from the formulas for the weights.
So assume that $2\le j \le m_{\ell}-1$.

For $\ell=2$, there is only the case $j=2$ left.
Direct computation shows 
\[ 
a_3^3=\frac25 < \frac23 = a_2^2.
\]

For $\ell\ge 3$, we use the absolutely convergent Fourier series
\[
u(x) \;:=\; \frac{\pi}{2} |\sin \pi x| \;=\; 
1 - 2 \sum_{k=1}^\infty \frac{\cos 2kx\pi}{4k^2-1} . 
\]
Using this for $x=\frac{j-1}{m_\ell-1}$ and the weight formula \eqref{eq:a} 
and abbreviating $u=u(x)$ we obtain
\[\begin{split}
(m_\ell-1) a_j^\ell - u \;&=\;  
- \frac{\cos (j-1)\pi}{(m_\ell-1)^2-1}  + 2\sum_{k=\frac{m_\ell-1}{2}}^{\infty}  
	\frac{ \cos \frac{2k(j-1) \pi}{m_\ell-1}}{4 k^2 - 1} \\
\;&=\;
\frac{\cos (j-1)\pi}{(m_\ell-1)^2-1}  + 2\sum_{k=\frac{m_\ell+1}{2}}^{\infty}  \frac{ \cos \frac{2k(j-1) \pi}{m_\ell-1}}{4 k^2 - 1.}
\end{split}\]
Note that
\[
\abs{2\, \sum_{k=\frac{m_\ell+1}{2}}^{\infty} 
					\frac{ \cos 2k \pi x}{4 k^2 - 1}}
\;\le\; \sum_{k=\frac{m_\ell+1}{2}}^{\infty} \frac{2}{4 k^2 - 1}
\;=\; \sum_{k=\frac{m_\ell+1}{2}}^{\infty} \Bigl(\frac{1}{2k - 1}
					- \frac{1}{2k + 1} \Bigr)
\;=\; \frac{1}{m_\ell}
\]
for all $x\in\R$. This implies
\[
\abs{(m_\ell-1)\, a_j^\ell - u} 
\;<\; \frac{1}{(m_\ell-1)^2-1} + \frac{1}{m_\ell}
\;=\; \frac{m_\ell-1}{m_\ell(m_\ell - 2)}  
\;<\; \frac{1}{m_\ell - 2} 
\]
and thus
\begin{equation} \label{eq:aju}
(m_\ell-1)\, a_j^\ell \;>\; u - \frac{1}{m_\ell-2}. 
\end{equation}
Similarly, noting that $x=\frac{j-1}{m_\ell-1}=\frac{2j-2}{m_{\ell+1}-1}$, 
we obtain that
\begin{equation} \label{eq:a2ju}
2 (m_\ell-1)\, a_{2j-1}^{\ell+1}
\;=\; (m_{\ell+1}-1)\, a_{2j-1}^{\ell+1} 
 \;<\; u + \frac{1}{m_{\ell+1}-2} 
 \;=\; u + \frac{1}{2 m_\ell-3} 
 \;<\; u + \frac{1}{2 (m_\ell-2)} . 
\end{equation}
Recall that our aim is to show that $a_{2j-1}^{\ell+1} < a_j^{\ell}$ which,
thanks to \eqref{eq:aju} and \eqref{eq:a2ju}, is certainly satisfied if
\[ 
u  \;>\; \frac{5}{2(m_\ell-2)}. 
\]
Using $m_\ell\ge 5$ and $\sin x \ge \frac{2\sqrt{2}x}{\pi}$ for $x \in[0,\pi/4]$,
we can conclude this from
\[
u \;=\; \frac{\pi}{2} \left|\sin \frac{(j-1)\pi}{m_\ell-1} \right| 
\;\ge\; \frac{\pi}{2} \sin \frac{\pi}{m_\ell-1} 
\;\ge\; \frac{\pi\sqrt{2}}{m_\ell-1}
\;>\; \frac{5}{m_\ell-1}
\;>\; \frac{5}{2(m_\ell-2)}. 
\]

This finally shows $a_{2j-1}^{\ell+1} < a_j^{\ell}$
for all $\ell\ge2$ and $j=1,\dots,m_\ell$ and, consequently, that 
$b_{j}^{\ell}<0$ for all $\ell\ge3$ and all odd $j$.

Now we can compute the norm
\[
\| \Delta^\ell \| \;=\; \sum_{j=1}^{m_\ell} |b_j^{\ell}| 
\;=\; \sum_{j=1}^{\frac{m_\ell-1}{2}} a_{2j}^{\ell} 
		- \sum_{j=1}^{\frac{m_\ell+1}{2}} a_{2j-1}^{\ell} 
		+ \sum_{j=1}^{m_{\ell-1}} a_{j}^{\ell-1} . 
\]
Using twice that the weights of one CC-rule add up to 1 we obtain
\[
\| \Delta^\ell \| = 2 \sum_{j=1}^{\frac{m_\ell-1}{2}} a_{2j}^{\ell} = 
2  \sum_{j=1}^{2^{\ell-2}}  \frac{1}{2^{\ell-1}} 
\left( 1 + \frac{1}{4^{\ell-1}-1}  - 
2\sum_{k=1}^{2^{\ell-2}-1} \frac{ \cos \frac{2k(2j-1) 
\pi}{2^{\ell-1}}}{4 k^2 - 1}\right). 
\]
Simplifying and changing the order of summation yields
$$ 
\| \Delta^\ell \| =  1 + \frac{1}{4^{\ell-1}-1} 
+ \frac{1}{2^{\ell-3}} \sum_{k=1}^{2^{\ell-2}-1} \frac{1}{4 k^2 - 1} 
\sum_{j=1}^{2^{\ell-2}}  \cos \frac{k(2j-1) \pi}{2^{\ell-2}}.
$$
Since the inner sum is always zero, we finally arrive at
$$ 
\| \Delta^\ell \| =  1 + \frac{1}{4^{\ell-1}-1.}
$$
\end{proof}

Now we are able to prove our explicit bound on the norm 
of the Smolyak algorithm. 

\begin{prop} \label{prop:norm}
For every $k\in\N_0$ and $d\in\N$ we have
\[
\norm{A(d+k,d)} \;\le\; \binom{d+k}{d} \;\le\; e^k\,\Bigl(1+\frac{d}{k}\Bigr)^k.
\]
\end{prop}

\begin{proof}
The second inequality is proven by
\[
\binom{d+k}{d} \;\le\; \frac{(d+k)^k}{k!} \;\le\; \left(\frac{e(d+k)}{k}\right)^k,
\]
where we used Stirling's approximation of the factorial.

To prove the first inequality, we show that 
\[
M(k,d) \;:=\; \max_{0\le\ell\le k}\, \norm{A(d+\ell,d)} \;\le\; \binom{d+k}{d}
\] 
by induction on $d$. 
For $d=1$ we obviously have 
\[ 
M(k,1) \;:=\; \norm{A(1+k,1)} =\norm{U^{k+1}} = 1 \le  \binom{1+k}{1}
\]
for every $k\in\N_0$. 
For $d\ge1$ let us now assume $M(k,d) \le \binom{d+k}{d}$ for all $k\in\N_0$ 
and recall that, for $d,k\in\N_0$, we have
\[
\sum_{\ell=0}^k\binom{d+\ell}{d}=\binom{d+k+1}{d+1}.
\]
Then we obtain by Lemma~\ref{lemma:Delta} that
\[\begin{split}
\|A(d+1+k,d+1)\|
\;&\le\; \sum_{\ell=1}^{k+1}\, \|A(d+1+k-\ell,d)\| \cdot \|\Delta^\ell\| \\
\;&\le\; \sum_{\ell=1}^{k+1}\, M(k+1-\ell,d)\, \|\Delta^\ell\|\\
\;&=\; M(k,d) \,+\, \frac23\, M(k-1,d) 
	+ \sum_{\ell=3}^{k+1}\, 
	\frac{4^{\ell-1}}{4^{\ell-1}-1}\; M(k+1-\ell,d) \\
\;&=\;  \sum_{\ell=1}^{k+1}\, M(k+1-\ell,d)
							\,-\, \frac13 M(k-1,d) \\
	&\qquad\qquad\qquad + \sum_{\ell=3}^{k+1}\, 
	\frac{1}{4^{\ell-1}-1}\; M(k+1-\ell,d) \\
\;&\le\;  \sum_{\ell=0}^{k}\, M(\ell,d)
\,-\, M(k-1,d)\, \left(\frac13 - \sum_{\ell=3}^{k+1}\, 
\frac{1}{4^{\ell-1}-1}\right), \\
\end{split}\]
where we have used that $M(k+1-\ell,d)\le M(k-1,d)$ for $k=3,\dots,k+1$.
Noting that 
\[
\sum_{\ell=3}^{k+1} \frac{1}{4^{\ell-1}-1} 
= \sum_{\ell=0}^{k-2} \frac{1}{4^{\ell+2}-1} 
\le \sum_{\ell=0}^{k-2} \frac{1}{4^{\ell+1}} \le \frac13,
\] 
we finally obtain
\[\begin{split}
\|A(d+1+k,d+1)\| \;&\le\; \sum_{\ell=0}^{k}\, M(\ell,d) 
\;\le\; \sum_{\ell=0}^{k}\, \binom{d+\ell}{d} \;=\; \binom{d+1+k}{d+1}  . 
\end{split}\]
Writing this inequality down with $k$ replaced by $\ell$ and 
taking the maximum over $\ell=0,1\dots,k$ on both sides 
leads to
\[
 M(k,d+1) \le \max_{0\le \ell\le k} \binom{d+1+\ell}{d+1} = \binom{d+1+k}{d+1}
\]
and concludes the induction step and the proof.\\
\end{proof}

Note that Proposition~1 holds also for the Smolyak algorithm that is based 
on other one-dimensional quadrature rules as long as 
\[ 
\|\Delta^1\| \;\le\; 1 \quad\text{ and }\quad 
\sum_{\ell=2}^\infty(\|\Delta^\ell\|-1) \;\le\; 0.
\]

To conclude our main result, Theorem~\ref{thm:main}, we additionally need 
a bound on the error of approximation by polynomials. 
We prove that $d$-dimensional $C^\infty$ functions with directional 
derivatives of all orders bounded by one can be arbitrarily well 
approximated by polynomials of total degree of order $\sqrt{d}$. 
This result was already proven by the authors and H. Wo\'zniakowski 
in \cite{HNUW13}, but we state the proof here again for completeness. 
Let 
$\mathcal{P}_k$ be the space of polynomials of degree $k$.

\begin{prop} \label{prop:poly}
Let $f\in F_d$ and $k\in\N$. Then
\[
\inf_{p\in\mathcal{P}_{k-1}}\|f-p\|_\infty 
\;\le\; \sqrt{\frac{1}{2\pi k}}\; \biggl( \frac{e \sqrt{d}}{2 k} \biggr)^k.
\]
\end{prop}

\begin{proof}
Consider the Taylor polynomial for $f\in F_d$ of order $k-1$ 
about the point $x^*=(1/2,\dots,1/2)$ which can be written as
$$  
T_{k-1}(x)=\sum_{\ell=0}^{k-1} \frac{f^{(\ell)}(x^*)(x-x^*)^\ell}{\ell!}
\ \ \ \ \mbox{for all}\ \ \ \ x\in [0,1]^d.
$$
Here we use the standard notation 
$A(x^\ell)=A(x,\dots,x)$ for the evaluation of an
$\ell$-linear map on the diagonal.
Note that we consider here $f^{(\ell)}(x^*)$ as an $\ell$-linear map.
It is well-known that the error of the approximation of $f$ by $T_{k-1}$ 
can be written as
$$
f(x)-T_{k-1}(x)\;=\; k\,\int_0^1(1-t)^{k-1}\,
\frac{ f^{(k)} \big(x^*+t(x-x^*)\big)(x-x^*)^{k}}{k!} \dint t,
$$
which implies
\[\begin{split}
\left| f(x)-T_{k-1}(x) \right| \;&\le\;  \frac{1}{(k-1)!}\, 
	\int_0^1(1-t)^{k-1}\, \dint t \, 
	\Bigl(\max_{\theta\in S^{d-1}}\,\|D^k_\theta f\|_\infty\Bigr)\, 
	\norm{x-x^*}_2^{k} \\
&\le\; \frac{1}{2^k\, k!}\, d^{k/2}.
\end{split}\]
An application of Stirling's formula proves the result 
since $T_{k-1}\in\mathcal{P}_{k-1}$.\\
\end{proof}

We combine Propositions~\ref{prop:norm} and \ref{prop:poly} 
to obtain the first part of our main result. 
To ease the calculations we use the trivial bound
$\inf_{p\in\mathcal{P}_{2k+1}}\|f-p\|_\infty\le 
\inf_{p\in\mathcal{P}_{2k-1}}\|f-p\|_\infty$. 
Recall that
\[
e(k,d) \;:=\; 
\sup_{f\in F_d} \abs{A\bigl(d+k,d\bigr)(f) \,-\,  S_d(f)} . 
\]
Using that $A(d+k,d)$ is exact for all polynomials of total 
degree at most $2k+1$, see~\cite{NR96,NR99}, we have
%
\[\begin{split}
e(k,d) \;&\le\; \sup_{f\in F_d} \inf_{p\in\mathcal{P}_{2k+1}}\|f-p\|_\infty\, 
\Bigl(1+\|A(d+k,d)\|\Bigr)\\
&\le\; \sqrt{\frac{1}{4\pi k}}\; 
				\biggl( \frac{e \sqrt{d}}{4k} \biggr)^{2k}\; 
				\biggl(1+e^k\,\Bigl(1+\frac{d}{k}\Bigr)^k \biggr) 
\;\le\; \sqrt{\frac{1}{\pi k}}\; 
\biggl( \frac{e \sqrt{d}}{4k} \biggr)^{2k}\; 
				e^k\,\Bigl(1+\frac{d}{k}\Bigr)^k  \\
&\le\; \frac{1}{\sqrt{\pi k}}\; \biggl( \frac{e^3}{16}\biggr)^k\;
				\biggl( \frac{d}{k^2} 
				\,+\, \frac{d^{2}}{k^3}\biggr)^k 
\;\le\; \frac{1}{\sqrt{\pi k}}\; \biggl( \frac{e^3 d}{8 k^2}\; 
				\max\Bigl\{ 1, \frac{d}{k} \Bigr\}\biggl)^k  . 
\end{split}\]

This shows that with
\[
k_{\eps,d} \;:=\; \biggl\lceil\max\Bigl\{3\,d^{2/3},\; \ln(1/\eps)\Bigr\} \biggr\rceil
\]
we have $e(k_{\eps,d},d)\le\eps$, which is the first part of Theorem~\ref{thm:main}.

It remains to prove the bound on the number of function values 
$N_d(k_{\eps,d})$ that are 
used by the Smolyak algorithm $A(d+k_{\eps,d},d)$. For this we use 
the following bound of Wasilkowski and Wo\'zniakowski.

\begin{lem}[\cite{WW95}, Lemma~7] \label{lemma:number}
The number of function values used by the Smolyak algorithm $A(d+k,d)$ that is 
based on the one-dimensional quadrature rules from \eqref{eq:U} satisfies 
\[
N_d(k) \;\le\; 2 (2e)^k \Bigl( 1+ \frac{d}{k} \Bigr)^k
\]
for all $k,d\in\N$, 
and thus
\[
\ln\bigl(N_d(k)\bigr) \;\le\; \ln2 \,+\, k\,\Bigr(\ln(2e) 
+ \ln\bigl( 1+ d/k \bigr)\Bigr).
\]
\end{lem}

\begin{proof}
Since our $m_\ell$'s satisfy $m_\ell\le2^\ell -1$ and the used nodes 
are nested, we use equation (43) of \cite{WW95} with $F_0=1$, $F=2$ and 
$q=d+k$. This gives 
\[
N_d(k) \;\le\; 2^{k+1} \binom{d+k-1}{d-1} 
\;\le\; 2 (2e)^k \Bigl( 1+ \frac{d}{k} \Bigr)^k
\]
by Stirling's formula.\\
\end{proof}

From this we obtain
\[\begin{split}
\ln\bigl(N_d(k_{\eps,d})\bigr) \;&\le\; \ln 2 +  k_{\eps,d}\, 
	\biggl( \ln(2e) \\
	&\qquad+ \min\Bigl\{ \ln\bigl( 1+(1/3)\, d^{1/3} \bigr),\, 
		\ln\bigl( 1+ d \ln(1/\eps)^{-1} \bigr) \Bigr\} \biggr) \\
\;&\le\; \ln 2 + k_{\eps,d}\, 
	\biggl( 2 + \ln\biggl( 1- \frac{d}{\ln(\eps)} \biggr) \biggr),
\end{split}\]
which is the second part of Theorem~\ref{thm:main}.

The proof of Theorem \ref{thm:qmc} proceeds similarly, so we only sketch the
necessary modifications. 
It follows from \cite{Dav67} that for all $k,d\in\N$ there exists
a cubature formula $Q(k,d)$ with positive weights, exactness $k$ and 
such that the number of function values $N^*_d(k)$ satisfies
the Tchakalov bound
\begin{equation} \label{eq:Tschakk}
 N^*_d(k) \le \binom{d+k}{d} \;\le\;  e^k \Bigl( 1+ \frac{d}{k} \Bigr)^k.
\end{equation}
This implies that $\norm{Q(k,d)}=1$ which can now be used instead of
Proposition \ref{prop:norm}. So, we obtain for the error 
$$
e(k,d) \;\le\; \sup_{f\in F_d} \inf_{p\in\mathcal{P}_{k}}\|f-p\|_\infty\, 
\Bigl(1+\|Q(k,d)\|\Bigr)
\le\; \sqrt{\frac{2}{\pi k}}\; \biggl( \frac{e \sqrt{d}}{2 k} \biggr)^k,
$$
which is smaller than $\eps$ for $k=k^*_{\eps,d}$ as in Theorem \ref{thm:qmc}.
Finally, the bound on $N^*_d(k^*_{\eps,d})$ in 
Theorem \ref{thm:qmc} now follows from \eqref{eq:Tschakk}.

%
%
%



\end{document}